\documentclass[11pt]{article}

\usepackage{amsmath}
\usepackage{graphicx}
\usepackage{subcaption}
\usepackage[top=30mm, bottom=30mm, left=27mm, right=27mm]{geometry}
\usepackage{amsfonts}
\usepackage{mathtools}
\usepackage{subcaption}
\usepackage{amssymb}
\usepackage{amsthm}
\usepackage{chngpage}
\graphicspath{ {Images/} }
\usepackage{hyperref}
\hypersetup{
	colorlinks=true,
	linkcolor=black,
	citecolor=black,
	filecolor=black,      
	urlcolor=black,
}
\usepackage{enumerate}
\usepackage{makeidx}
\usepackage{bm}
\usepackage{thmtools}
\usepackage{thm-restate}

\makeindex

\makeatletter
\renewenvironment{proof}[1][\proofname] {\par\pushQED{\qed}\normalfont\topsep6\p@\@plus6\p@\relax\trivlist\item[\hskip\labelsep\bfseries#1\@addpunct{.}]\ignorespaces}{\popQED\endtrivlist\@endpefalse}
\makeatother

\newcommand{\gap}{\operatorname{gap}}
\newcommand{\CC}{C^{\textnormal{compl}}}
\newcommand{\CCC}{\mathcal{C}^{<}}
\newcommand{\ZZ}{\mathbb{Z}}
\newcommand{\ZZZ}{\mathbb{Z}_{\geq 0}}
\newcommand{\thth}{\textsuperscript{th} }

\newtheorem{proposition}{Proposition}
\newtheorem{lemma}[proposition]{Lemma}
\newtheorem{theorem}[proposition]{Theorem}

\newtheorem{question}[proposition]{Question}
\usepackage{amsthm}

\theoremstyle{definition}

\newtheorem*{remark*}{Remark}
\newtheorem*{theorem*}{Theorem}

\title{Projections of antichains}
\author{Barnab\'{a}s Janzer\thanks{Department of Pure Mathematics and Mathematical Statistics, Centre for Mathematical Sciences, University of Cambridge, Wilberforce Road, Cambridge CB3 0WB, United Kingdom. Email: bkj21@cam.ac.uk}}
\date{\vspace{-21pt}}

\begin{document}
	\maketitle
	
\begin{abstract}
	A subset $A$ of $\mathbb{Z}^n$ is called a weak antichain if it does not contain two elements $x$ and $y$ satisfying $x_i<y_i$ for all $i$. Engel, Mitsis, Pelekis and Reiher showed that for any weak antichain $A$, the sum of the sizes of its $(n-1)$-dimensional projections must be at least as large as its size $|A|$. They asked what the smallest possible value of the gap between these two quantities is in terms of  $|A|$. We answer this question by giving an explicit weak antichain attaining this minimum for each possible value of $|A|$. In particular, we show that sets of the form $A_N=\{x\in\ZZ^n: 0\leq x_j\leq N-1\textnormal{ for all $j$ and } x_i=0\textnormal{ for some $i$}\}$ minimise the gap among weak antichains of size $|A_N|$.
\end{abstract}\vspace{-5pt}	

\section{Introduction}
A subset of $\ZZ^n$ is called a \textit{weak antichain} if it contains no elements $x$ and $y$ such that for all $i$ $x_i<y_i$. Let us denote by $\pi_i$ the projection along the $i$\thth coordinate, that is, $\pi_i: \ZZ^{n}\to\ZZ^{n-1}$ is given by $(x_1\dots,x_n)\mapsto (x_1,\dots,x_{i-1},x_{i+1},\dots,x_n)$. Engel, Mitsis, Pelekis and Reiher \cite{engel2018projection} proved the following projection inequality for weak antichains (which they used to prove an analogous result about weak antichains in the continuous cube $[0,1]^n$).
\begin{theorem}[Engel, Mitsis, Pelekis and Reiher \cite{engel2018projection}]
	For every finite weak antichain $A$ in $\ZZ^n$, we have
	$$|A|\leq \sum_{i=1}^{n}|\pi_i(A)|.$$
\end{theorem}
The same authors asked the following question.
\begin{question}\label{question_weakantichains}
	What is the smallest possible value $g(n,m)$ of $\gap(A)=\sum_{i=1}^{n}|\pi_i(A)|-|A|$ as $A$ varies over weak antichains in $\ZZ^n$ of size $m$?
\end{question}

Note that the question is uninteresting for (strong) antichains $A$ in $\ZZ^n$, as we trivially have $|\pi_i(A)|=|A|$ for all $i$ in this case. Furthermore, a weak antichain in $\{0,1\}^n$ is just a subset of $\{0,1\}^n$ not containing both the zero vector and the vector with all entries equal to 1. So classical results about set systems (such as Sperner's theorem, see e.g. \cite{bollobas1986combinatorics}) are not particularly relevant here.\medskip

In this paper we answer Question \ref{question_weakantichains}. To state the result, we need some definitions. Let $\ZZZ$ denote the set of non-negative integers, and let $X_n$ be the subset of $\ZZZ^n$ consisting of elements that have at least one coordinate which is zero. Note that any subset of $X_n$ is a weak antichain. For given $x, y\in X_n$, let $T=\{i: x_i\not = y_i \}$, let $x'=(x_i)_{i\in T}$, $y'=(y_i)_{i\in T}$. Write $x<y$ if $\max x'<\max y'$ or ($\max x'=\max y'$ and $\max\{i: x'_i=\max x'\}<\max\{i: y'_i=\max y'\}$). Then $<$ defines a total order on $X_n$. We will call this the \textit{balanced order} on $X_n$.
\begin{restatable}{theorem}{initialsegments}\label{theorem_initialsegments}
	For every $n\geq 2$ and $m\geq 0$, the initial segment of size $m$ of the balanced order on $X_n$ minimises the gap among weak antichains in $\ZZ^n$ of size $m$. In particular, for every  positive integer $N$, the set
	$$A_N=\{x\in \ZZZ^n: 0\leq x_i\leq N-1\textnormal{ for all $i$, and $x_j=0$ for some $j$}\}$$
	minimises the gap among weak antichains of size $|A_N|=N^n-(N-1)^n$.
\end{restatable}

In terms of asymptotic lower bounds on the gap, this gives the following result.

\begin{restatable}{theorem}{explicit}\label{theorem_explicit}
	For every $n\geq 2$ and $m\geq 0$, we have
	$$g(n,m)\geq c_nm^{1-1/(n-1)},$$ where $c_n=\frac{1}{2}(n-1)n^{1/(n-1)}$.
	Moreover, for every $n\geq 2$, we have $$g(n,m)\sim c_nm^{1-1/(n-1)} \textnormal{ as }m\to\infty.$$
\end{restatable}

Our proofs have the following structure. Starting with any weak antichain, we modify it into a subset of $X_n$. This modification will be made step-by-step, and at some points our set will not be a weak antichain. However, it will always satisfy a certain weaker property, which we will call 'layer-decomposability'. Studying subsets of $X_{n}$ is much simpler than studying general weak antichains, and we will finish the proof of Theorem \ref{theorem_initialsegments} using induction on $n$ and codimension-1 compressions. For our proof to work we will need to show that initial segments of the balanced order are extremal for another property as well. Instead of deducing the asymptotic result from Theorem \ref{theorem_initialsegments}, we will prove it directly and before Theorem \ref{theorem_initialsegments}, because its proof is simpler and motivates some of the steps in the proof of Theorem \ref{theorem_initialsegments}.

\section{Compressing to a down-set in $X_n$}
Recall that we denote $X_n=\{x\in \ZZZ^n: x_i=0\textnormal{ for some $i$} \}$. In this section our aim is to prove the following lemma. 

\begin{lemma}\label{lemma_downset}
	If $A$ is a finite weak antichain in $\mathbb{Z}^n$, then there is a weak antichain $A'\subseteq X_n$ of the same size having $|\pi_i(A')|\leq |\pi_i(A)|$ for each $i$ which is a down-set, i.e., if $x, y\in \ZZZ^n$, $x_i\leq y_i$ for all $i$ and $y\in A'$ then $x\in A'$.
\end{lemma}

We start by recalling the proof of Engel, Mitsis, Pelekis and Reiher \cite{engel2018projection} that $\gap(A)\geq 0$ for every finite weak antichain. For any finite set $A\subseteq \ZZ^n$, define the $i$\thth bottom layer $B_i(A)$ to be the set of elements with minimal $i$\thth coordinate, i.e., $$B_i(A)=\{x\in A: \textnormal{whenever $y\in A$ with $y_j=x_j$ for all $j\not=i$ then $y_i\geq x_i$}\}.$$
Furthermore, define $A_1,\dots,A_n$ inductively by setting ($A_1=B_1(A)$ and)
$$A_{i}=B_i\left(A\setminus(A_1\cup\dots\cup A_{i-1})\right).$$
Observe that for a weak antichain we have $A=A_1\cup\dots\cup A_n$. Indeed, if $x\in A\setminus (A_1\cup\dots\cup A_n)$ then we may inductively find $x^{(i)}\in A_{n-i}$ (for $0\leq i\leq n-1$) such that $x^{(i)}_j<x_j$ for $j\geq n-i$ and $x^{(i)}_j=x_j$ for $j<n-i$. Then $x^{(n-1)}$ has all coordinates smaller than $x$, giving a contradiction.

We will call a finite set $A$ with $A=A_1\cup\dots\cup A_n$ \textit{layer-decomposable}. Note that $\pi_i$ restricted to $A_i$ is injective, hence $\sum_{i=1}^n|\pi_i(A)|\geq \sum_{i=1}^n|A_i|=|A|$ for layer-decomposable sets (and in particular for weak antichains).

Now assume $A\subseteq \ZZZ^n$. Define the \textit{$i$-compression} $C_i(A)$ of $A$ by replacing each $x\in B_i(A)$ by $(x_1,\dots, x_{i-1},0,x_{i+1},\dots,x_n)$. Note that $|C_i(A)|=|A|$.

\begin{lemma}\label{lemma_compression}
	Let $A\subseteq \ZZZ^n$ be any finite set. For every $i\not=j$, $\pi_j(C_i(A))\subseteq C_i(\pi_j(A))$. In particular, $|\pi_j(C_i(A))|\leq |\pi_j(A)|$.
\end{lemma}

	(When considering $C_i(\pi_j(A))$, we mean compressing along the coordinate labelled by $i$, not along the $i$\thth remaining coordinate.)

\begin{proof}
	Suppose $(x_1,\dots, x_{j-1},x_{j+1},\dots,x_n)\in \pi_j(C_i(A))$ so that there is an $x\in C_i(A)$ with $k$\thth coordinate $x_k$ for all $k$.
	\begin{itemize}
		\item If $x_i=0$ then there is a $y\in B_i(A)$ with $x_k=y_k$ for $k\not =i$. So $(y_1,\dots, y_{j-1},y_{j+1},\dots,y_n)\in \pi_j(A)$. But this vector agrees with $(x_1,\dots, x_{j-1},x_{j+1},\dots,x_n)$ in all entries except maybe the one labelled by $i$, so (since $x_i=0$) we have $(x_1,\dots, x_{j-1},x_{j+1},\dots,x_n)\in C_i(\pi_j(A))$.
		\item If $x_i\not =0$ then there is an $x\in A\setminus B_i(A)$ with $k$\thth coordinate $x_k$ for all $k$, and $y\in B_i(A)$ with $y_k=x_k$ for $k\not =i$ and $y_i<x_i$. But then $\pi_j(y)$ and $\pi_j(x)$ agree in all coordinates except the $i$\thth one, which shows $\pi_j(x)\not \in B_i(\pi_j(A))$ and hence $\pi_j(x)\in C_i(\pi_j(A))$ as claimed.
	\end{itemize}
\end{proof}

Say that $A$ is $i$-compressed if $C_i(A)=A$, i.e., $B_i(A)=\{x\in A: x_i=0\}$.
\begin{lemma}\label{lemma_compressionstep}
	Suppose that $A\subseteq \ZZZ^n$ is finite, layer-decomposable (i.e., $A=A_1\cup\dots\cup A_n$), and $k$-compressed for all $k<i$. Then $A'=C_i(A)$ satisfies
	\begin{enumerate}[(i)]
		\item $A'$ is $k$-compressed for all $k\leq i$.\label{lemmapart_compressioni}
		\item $A'$ is layer-decomposable
		.\label{lemmapart_compressionii}
	\end{enumerate}
\end{lemma}
\begin{proof}
	Let $j<i$. By Lemma \ref{lemma_compression}, $|\pi_j(A')|\leq |\pi_j(A)|$. But, since $B_j(A)=\{x\in A: x_j=0\}$, $C_i(B_j(A))$ is a subset of $A'$ having $j$\thth coordinate constant zero and $j$\thth projection of size $|\pi_j(A)|$. It follows that $B_j(A')=C_i(B_j(A))=\{x\in A': x_j=0\}$, giving (\ref{lemmapart_compressioni}).
	
	We now show (\ref{lemmapart_compressionii}). Since $A$ is $k$-compressed for all $k<i$, easy induction on $k$ gives
	\[A_k=\{x\in A: x_k=0\textnormal{ but $x_l\not =0$ for $l<k$} \}\hspace{12pt}\textnormal{for all $k<i$},\]
	and similarly
	\[A'_k=\{x\in A': x_k=0\textnormal{ but $x_l\not =0$ for $l<k$} \}\hspace{12pt}\textnormal{for all $k\leq i$}.\]
	But then we have 
	\begin{align*}C_i(A\setminus (A_1\cup \dots\cup A_{i-1}))&=C_i(\{x\in A: x_k\not=0\textnormal{ for all $k<i$} \})\\
	&=\{x\in C_i(A): x_k\not=0\textnormal{ for all $k<i$} \}\\
	&=A'\setminus (A'_1\cup \dots\cup A'_{i-1}).
	\end{align*}
	It follows that $A'\setminus(A'_1\cup \dots\cup A'_i)=\{x\in A'\setminus (A'_1\cup \dots\cup A'_{i-1}): x_i\not =0\}=A\setminus(A_1\cup \dots\cup A_{i})$. Using $A=A_1\cup\dots\cup A_n$, we get (\ref{lemmapart_compressionii}).
\end{proof}

\begin{lemma}\label{lemma_compressedset}
	If $A\subseteq \ZZZ^n$ is a finite weak antichain, then $A'=C_n(C_{n-1}(\dots(C_1(A))\dots))$ satisfies
	\begin{enumerate}[(i)]
		\item $|\pi_i(A')|\leq |\pi_i(A)|$ for each $i$.\label{lemmapart_resultsofcompressioni}
		\item $A'$ is $k$-compressed for all $k$. \label{lemmapart_resultsofcompressionii}
		\item $A'=A'_1\cup\dots\cup A'_n$.\label{lemmapart_resultsofcompressioniii}
		\item $A'_k=\{x\in A': x_k=0\textnormal{ but $x_l\not =0$ for $l<k$} \}$ for all $k$.\label{lemmapart_resultsofcompressioniv}
		\item $A'\subseteq X_n=\{x\in \ZZZ^n: x_i=0\textnormal{ for some $i$}\}$.\label{lemmapart_resultsofcompressionv}
	\end{enumerate}
\end{lemma}
\begin{proof}
	All of these are immediate from Lemma \ref{lemma_compression} and Lemma \ref{lemma_compressionstep}.
\end{proof}

Note that even though some intermediate steps $C_i(C_{i-1}(\dots(C_1(A))\dots))$ need not give weak antichains, we see that after the $n$\thth compression we end up with a set which is necessarily a weak antichain.\medskip

For a set $A\subseteq \ZZZ^n$, define the \textit{complete $i$-compression} \begin{align*}\CC_{i}(A)=&\{(x_1,\dots,x_{i-1},a,x_{i+1},\dots,x_n):\\ &\textnormal{$a\in \ZZZ$ and there are at least $a+1$ elements $y$ of $A$ having for all $j\not =i$ $y_j=x_j$} \}.\end{align*}
Note that $|\CC_i(A)|=|A|$.
\begin{lemma}\label{lemma_completecompression}
	If $A\subseteq X_n$ then for any $j$ we have $|\pi_j(\CC_i(A))|\leq |\pi_j(A)|$.
\end{lemma}
\begin{proof}
	The proof is essentially the same as for Lemma \ref{lemma_compression}. Indeed, let $j\not =i$ and suppose that $(x_1,\dots,x_{j-1},x_{j+1},\dots,x_n)\in \pi_j(\CC_i(A))$. So there is an $x\in \CC_i(A)$ with $k$\thth coordinate $x_k$ for all $k$, and hence there are $y^{(0)}, \dots, y^{(x_i)}\in A$ such that $y^{(a)}_k=x_k$ for all $k\not =i$ and all $0\leq a\leq x_i-1$, and $y^{(0)}_i<y^{(1)}_i<\dots<y^{(x_i)}_i$. But then $(x_1,\dots,x_{j-1},x_{j+1},\dots,x_n)\in \CC_i(\pi_j(A))$. It follows that $\pi_j(\CC_i(A))\subseteq \CC_i(\pi_j(A))$, giving the result.
	
	[Alternatively, we can deduce Lemma \ref{lemma_completecompression} from Lemma \ref{lemma_compression} by applying $C_i$ to $A$ then $A\setminus B_i(A)$ then $A\setminus(B_i(A)\cup B_i(A\setminus B_i(A)))$ and so on.]
\end{proof}


\begin{proof}[Proof of Lemma \ref{lemma_downset}]
	We may assume that $A\subseteq \ZZZ^n$. By Lemma \ref{lemma_compressedset} we may also assume $A\subseteq X_n$. Keep applying complete compressions while it changes our set. These do not increase any projection by Lemma \ref{lemma_completecompression}, and keeps our set a subset of $X_n$. Note that if $A\not =\CC_i(A)$ then $\sum_{x\in \CC_i(A)}\sum_jx_j<\sum_{x\in A}\sum_jx_j$, so the process must terminate. So the set $A'$ we end up with must have $\CC_{i}(A')=A'$ for all $i$, so it must be a down-set.
\end{proof}

\section{The asymptotic result}
We now show how Lemma \ref{lemma_downset} can be used to prove the asymptotic version of our theorem. The proof of the exact version (Theorem \ref{theorem_initialsegments}) in the next section will be independent of this section, but the proof below motivates some of the steps in the proof of Theorem \ref{theorem_initialsegments}. Recall that $g(n,m)$ denotes the smallest possible value of $\gap(A)=\sum_{i=1}^{n}|\pi_i(A)|-|A|$ as $A$ varies over weak antichains of size $m$ in $\ZZ^n$, and our aim is to prove the following result.
\explicit*

\begin{proof}
	By Lemma \ref{lemma_downset}, it suffices to consider sets $A\subseteq X_n$ which are down-sets. We prove the result by induction on $n$. The case $n=2$ is trivial, now assume $n\geq 3$ and the result holds for $n-1$. Clearly, we may assume $m\geq 1$.
	
	Define, for every $a\in\ZZZ$,
	$$L_a=\{(x_1,x_2,\dots,x_{n-1})\in\ZZZ^{n-1}: (x_1,x_2,\dots,x_{n-1},a)\in A\textnormal{ and $x_i\not =0$ for some $i<n$}\}.$$
	
	Let $K=\pi_n(B_n(A))\setminus L_0$. Note that $A$ can be written as a disjoint union of $K\times\{0\}$ and the $L_i\times\{i\}$. Also, $L_0\supseteq L_1\supseteq L_2\supseteq \dots$. Note furthermore that $|\pi_i(A)|=\sum_{j\geq 0}|\pi_i(L_j)|$ for $i<n$. It follows that
\begin{align*}\sum_{i=1}^{n}|\pi_i(A)|-|A|&=\sum_{i=1}^{n-1}|\pi_i(L_0)|+\sum_{j\geq 1}\left(\sum_{i=1}^{n-1}|\pi_i(L_j)|-|L_j|\right)\\
&\geq |L_0|+\sum_{j\geq 0}g(n-1,|L_j|)\\
&\geq |L_0|+\sum_{j\geq 0}c_{n-1}|L_j|^{1-1/(n-2)}.
\end{align*}
Write $|L_0|=x$. Since $0\leq |L_j|\leq x$ for each $j$, we have $|L_j|^{1-1/(n-2)}\geq \frac{|L_j|}{x}x^{1-1/(n-2)}.$
It follows that 
$$\sum_{i=1}^{n}|\pi_i(A)|-|A|\geq x+c_{n-1}\left(\sum_{j\geq 0}|L_j|\right)x^{-1/(n-2)}.$$

Note that $\sum_{j\geq 0} |L_i|=m-|K|$. By the (discrete) Loomis--Whitney inequality \cite{loomis1949inequality} (see \cite{bollobas1995projections} for a generalisation),
\[|K|^{n-2}\leq \prod_{i=1}^{n-1}|\pi_i(K)|\leq \left(\frac{\sum_{i=1}^{n-1}|\pi_i(K)|}{n-1} \right)^{n-1}.\]
But $\sum_{i=1}^{n-1}|\pi_i(K)|\leq |L_0|$ since we may assign to $(x_1,\dots,x_{i-1},x_{i+1},\dots,x_{n-1})\in \pi_i(K)$ the value $(x_1,\dots,x_{i-1},0,x_{i+1},\dots,x_{n-1})\in L_0$, giving an injective function from the disjoint union of the projections to $L_0$. It follows that
$$|K|^{n-2}\leq \left(\frac{x}{n-1}\right)^{n-1}$$
and so
\begin{align*}\sum_{i=1}^{n}|\pi_i(A)|-|A|&\geq x+c_{n-1}\left(m-\frac{1}{(n-1)^{1+1/(n-2)}}x^{1+1/(n-2)}\right)x^{-1/(n-2)}\\
&=\left(1-\frac{c_{n-1}}{(n-1)^{1+1/(n-2)}}\right)x+c_{n-1}mx^{-1/(n-2)}.
\end{align*}

Differentiation shows that this is minimised at
$$x=\left( \frac{c_{n-1}m}{(n-2)\left(1-\frac{c_{n-1}}{(n-1)^{1+1/(n-2)}}\right)}\right)^{1-1/(n-1)}$$
giving 
\begin{align*}\sum_{i=1}^{n}|\pi_i(A)|-|A|&\geq
\left(1-\frac{c_{n-1}}{(n-1)^{1+1/(n-2)}}\right)^{1/(n-1)}(n-1)(n-2)^{1/(n-1)-1}(c_{n-1}m)^{1-1/(n-1)}.
\end{align*}

But $c_{n-1}=\frac{1}{2}(n-2)(n-1)^{1/(n-2)}$, so
\begin{align*}1-\frac{c_{n-1}}{(n-1)^{1+1/(n-2)}}=\frac{n}{2(n-1)}
\end{align*}
and so
$$\left(1-\frac{c_{n-1}}{(n-1)^{1+1/(n-2)}}\right)^{1/(n-1)}(n-1)(n-2)^{1/(n-1)-1}c_{n-1}^{1-1/(n-1)}=\frac{1}{2}(n-1)n^{1/(n-1)}=c_n,$$
giving $g(n,m)\geq c_nm^{1-1/(n-1)}$, as claimed.

It remains to show that for any fixed $n$ we have $g(n,m)\leq (1+o(1))c_nm^{1-1/(n-1)}$. Let
$$A_N=\{(x_1,\dots,x_N):\textnormal{ $0\leq x_i\leq N-1$ for all $i$, and there is a $j$ such that $x_j=0$}\}.$$

Note that $A_N$ has $|\pi_i(A_N)|=N^{n-1}$ for each $i$, so
$$\sum_{i=1}^{n}|\pi_i(A_N)|=nN^{n-1}.$$ Moreover, it has size $$m_N=|A_N|=N^n-(N-1)^n=nN^{n-1}-\binom{n}{2}N^{n-2}+O(N^{n-3}).$$

Now pick $N$ such that $m_N\leq m<m_{N+1}$, and consider the weak antichain given as follows. Let $B$ be an arbitrary subset of $\{0\}\times [N,N+\lfloor(m_{N+1}-m_N)^{1/(n-1)}\rfloor]^{n-1}$ of size $m-m_N$. Note that $B$ has gap at most
$$(n-1)(m_{N+1}-m_N+1)^{(n-2)/(n-1)}=O(N^{(n-2)^2/(n-1)})$$
Put $A=A_N\cup B$. So $A$ has size $m$ and gap equal to the sum of gaps of $A_N$ and $B$, so $A$ has gap at most $$\binom{n}{2}N^{n-2}+O(N^{n-3})+O(N^{(n-2)^2/(n-1)}).$$

But $m=nN^{n-1}(1+o(1))$, so the gap is $c_nm^{1-1/(n-1)}(1+o(1))$, as required.
\end{proof}

\section{The exact result}

Recall that we defined a total order (called the balanced order) on $X_n$ as follows. Given $x, y\in X_n$, let $T=\{i: x_i\not = y_i \}$, let $x'=(x_i)_{i\in T}$, $y'=(y_i)_{i\in T}$. Write $x<y$ if $\max x'<\max y'$ or ($\max x'=\max y'$ and $\max\{i: x'_i=\max x'\}<\max\{i: y'_i=\max y'\}$). To see that this really is a total order, we need to show that if $x<y$ and $y<z$, then $x<z$. Set $M_x=\max x$ and $i_x=\max\{i: x_i=M_x\}$, and define $M_y, M_z, i_y, i_z$ similarly. If $M_x<M_y$ or $M_y<M_z$, it is clear that $x<z$. If $M_x=M_y=M_z$ and either $i_x<i_y$ or $i_y<i_z$, $x<z$ again follows. Finally, if $M_x=M_y=M_z$ and $i_x=i_y=i_z$ then $x<z$ follows from induction on $n$. Recall that the result we are trying to prove is the following.

\initialsegments*

If $A\subseteq X_n$
, we define the balanced-$i$-compression $\CCC_i(A)$ as follows. For each $a$, write
$$L^i_a(A)=\{(x_1,\dots, x_{i-1},x_{i+1},\dots,x_n)\in X_{n-1}: (x_1,\dots, x_{i-1},a,x_{i+1},\dots,x_n)\in A \}.$$
Also write
\[K^i(A)=
\{(x_1,\dots,x_{i-1},x_{i+1},\dots,x_n)\in \ZZ_{>0}^n: (x_1,\dots,x_{i-1},0,x_{i+1},\dots,x_n)\in A\}.\]
(Here $\ZZ_{>0}$ denotes the set of positive integers.) We define $A'=\CCC_i(A)$ by letting $L^i_a(A')$ be the initial segment of the balanced order on $X_{n-1}$ of size $|L_i(A)|$ for each $a$, and letting $K^i(A')$ be the first $|K^i(A)|$ elements of the ordering $\prec$ on $\mathbb{Z}_{>0}^{n-1}$ given by $(x_1,\dots,x_{i-1},x_{i+1},\dots,x_n)\prec(y_1,\dots,y_{i-1},y_{i+1},\dots,y_n)$ if and only if $(x_1,\dots,x_{i-1},0,x_{i+1},\dots,x_n)<(y_1,\dots,y_{i-1},0,y_{i+1},\dots,y_n)$ (in the balanced order) on $X_n$. Observe that $|A'|=|A|$.

For these balanced-$i$-compressions to be useful, we will have to establish another extremal property of initial segments.
For $A\subseteq X_n$, write
$$S(A)=\{(x_1,\dots,x_n)\in \mathbb{Z}_{>0}^n: \textnormal{ for all $k$ we have } (x_1,\dots,x_{k-1},0,x_{k+1},\dots,x_n)\in A \}.$$

\begin{lemma}\label{lemma_<compression}
	Let $n\geq 3$. Suppose that initial segments $I$ of the balanced order maximise $|S(I)|$ among down-sets in $X_{n-1}$ of given size. Then whenever $A$ is a down-set in $X_{n}$ and $i\in\{1,\dots,n\}$, then $A'=\CCC_i(A)$ satisfies the following.
	\begin{enumerate}[(i)]
		\item $A'$ is a down-set.
		\item $|S(A')|\geq |S(A)|$.
		\item If it is also true that initial segments of the balanced order minimise the gap among subsets of $X_{n-1}$ of given size, then $\gap(A')\leq \gap(A)$.
	\end{enumerate}
\end{lemma}
\begin{proof}
	(i) It is clear that $L_0^i(A')\supseteq L_1^i(A')\supseteq\dots$, and that the $L_a^i(A')$ and  $K^i(A')$ are down-sets, so it remains to show that $K^i(A')\subseteq S(L^i_0(A'))$. Note that we know this is true for $A$ instead of $A'$ since $A$ is a down-set. Also, it is easy to see that if $x,y\in \mathbb{Z}_{>0}^{n-1}$, $I$ is an initial segment of the balanced order on $X_{n-1}$, $x\prec y$ in the ordering $\prec$ of $\mathbb{Z}_{>0}^{n-1}$ defined earlier and $y\in S(I)$ then $x\in S(I)$. Indeed, if $T=\{j: x_j\not=y_j\}$ and $k=\min\{l\in T: y_l=\min_{j\in T}y_j\}$, then we have the following. \newline If $j\not \in T$ then $(x_1,\dots,x_{j-1},0,x_{j+1},\dots,x_{n-1})<(y_1,\dots,y_{j-1},0,y_{j+1},\dots,y_{n-1})$.\newline If $j\in T$ then $(x_1,\dots,x_{j-1},0,x_{j+1},\dots,x_{n-1})\leq(y_1,\dots,y_{k-1},0,y_{k+1},\dots,y_{n-1})$.
	
	So $S(L_0^i(A'))$ and $K^i(A')$ are both initial segments. But $|S(L_0^i(A'))|\geq |S(L_0^i(A))|\geq |K^i(A)|=|K^i(A')|$, proving (i).
	
	(ii) Note that for any $B\subseteq X_n$ and $a>0$ we have
	$$\{(x_1,\dots,x_{i-1},x_{i+1},\dots,x_n)\in\mathbb{Z}_{>0}^{n-1}: (x_1,\dots,x_{i-1},a,x_{i+1},\dots,x_n)\in S(B)\}=K^i(B)\cap S(L^i_a(B)).$$
	But for each $a$ we have $|S(L^i_a(A))|\leq |S(L^i_a(A'))|$, and $K^i(A'), S(L^i_a(A'))$ are nested. This implies that \begin{align*}|K^i(A)\cap S(L^i_a(A))|&\leq\min(|K^i(A)|,|S(L^i_a(A))|)\\&\leq\min(|K^i(A')|,|S(L^i_a(A'))|)=|K^i(A')\cap S(L^i_a(A'))|,\end{align*} proving (ii).
	
	(iii) For any down-set $B\subseteq X_n$ we have $\gap(B)=|L^i_0(B)|+\sum_{a\geq 0}\gap(L^i_a(B))$. But then (iii) follows trivially from the assumption that initial segments minimise $\gap$ on $X_{n-1}$.	
\end{proof}

\begin{lemma}\label{lemma_<compressedproperty}
	Suppose $n\geq 3$ and $A\subseteq X_n$ is a down-set having $\CCC_i(A)=A$ for all $i$. Assume that $x<y$ with $x\not \in A$ and $y\in A$. Then
	\begin{enumerate}[(i)]
		\item $x$ has a unique coordinate which is zero.
		\item if $x_j=y_j$ for some $j$, then $x_j=y_j=0$ and $y$ has at least one other coordinate which is zero.
	\end{enumerate}
\end{lemma}

\begin{proof}
	Observe the following. If $x_l=y_l$ for some $l$, then $\CCC_l(A)=A$ shows that in fact it must be the case that $x_l=y_l=0$ and exactly one of $x,y$ have a zero coordinate not at the $i$\thth position. It follows that if we write $i=\max\{j: y_j=\max y\}$ then (no longer assuming $x_l=y_l$)
	\begin{align*}
	y_i&\geq x_j\textnormal{ for all $j$, and}\\
	y_i&> x_j\textnormal{ for all $j\geq i$.}
	\end{align*}
	Assume again that $x_l=y_l$ some $l$, necessarily $l\not =i$. Pick some $k\not =i, l$. Then the vector $y'$ obtained by replacing the $k$\thth coordinate of $y$ by $0$ is in $A$ (since $A$ is a down-set), and we have $y'>x$. By the same argument as above, we deduce that $x_l=y'_l=0$, and -- since $y'_k=0$ -- it must be the case that $x$ has no zero coordinates other than the $l$\thth one. We deduce that whenever $x_l=y_l$ for some $l$, then $x_l=y_l=0$, $x_s\not =0$ for $s\not =l$, and there is an $s\not =l$ such that $y_s=0$. No longer assuming $x_l=y_l$, this also shows that $x$ has at most one (so exactly one) zero coordinate: if $x_k=x_l=0$ ($k\not=l$), then we may assume that $k\not =i$ and then the vector $y'$ obtained by replacing the $k$\thth coordinate of $y$ by $0$ contradicts the observations above.
\end{proof}

\begin{lemma} For every $n\geq 2$, initial segments $I$ of the balanced order maximise $|S(I)|$ among down-sets in $X_n$ of given size.
\end{lemma}
\begin{proof}
	We prove the statement by induction on $n$. The case $n=2$ is trivial, now assume $n\geq 3$ and the result holds for smaller values of $n$.
	Let $A$ be any subset of $X_n$, we show the initial segment $I$ of same size has $|S(I)|\geq |S(A)|$. Taking a down-set $A'$ in $X_n$ minimising $\sum_{x\in A'}\textnormal{(position of $x$ in the balanced order)}$ among sets with $|A'|=|A|$ and $|S(A')|\geq |S(A)|$, we may assume that $\CCC_i(A)=A$ for each $i$ (by Lemma \ref{lemma_<compression}). Suppose there are $x,y\in X_n$ with $x<y$, $y\in A$ and $x\not \in A$.
	
	Take $y$ to be maximal (in the balanced order). Let $i=\max\{j: y_j=\max y\}$. If there is an $x\not\in A$ with $x<y$ and the unique zero coordinate not being at the $i$\thth position, pick the minimal of these (in the balanced order). Otherwise pick $x\not\in A$ which is minimal. Consider $A'=A\setminus\{y\}\cup\{x\}$. Note that $A'$ is again a down-set.
	
	We show that $|S(A')|\geq |S(A)|$. (This would give a contradiction.) If $y$ has more than one zero coordinates, then $S(A)\setminus S(A')=\emptyset$, so the claim is clear. Otherwise we must have $x_l\not =y_l$ for all $l$ by Lemma \ref{lemma_<compressedproperty}, and $y$ has a unique zero coordinate $y_t$. Observe that
	\begin{align*}S(A)\setminus S(A')=\{&(y_1,\dots,y_{t-1},a,y_{t+1},\dots,y_n):\\ &a\in\mathbb{Z}_{>0}\textnormal{ and replacing any coordinate by $0$ we get an element of $A$}\}.
	\end{align*}
	Recall that there is a unique $s$ such that $x_s=0$. We claim that $S(A)\setminus S(A')$ is empty unless $s=i$. Indeed, suppose $s\not =i$ and $S(A)\setminus S(A')$ has an element $z$ corresponding to $a\geq1$. 
	Let $z'$ be obtained from $z$ by setting the $s$\thth coordinate to be zero. Then $z'\in A$, $z'>x$, $x_s=z'_s=0$ and there is a unique coordinate at which $z'$ is zero, contradicting Lemma \ref{lemma_<compressedproperty}.

	So we may assume $s=i$.
	Note that if $a\geq x_t$ and the corresponding vector appears in the set above, then $A$ has an element $z$ with $z_i=y_i$ and $z_t=x_t\not =0$ (using that $n\geq 3$ and that $A$ is a down-set. Note that $x_t\not =0$ since $x_l\not =y_l$ for all $l$.) But then $z>x$, so this contradicts Lemma \ref{lemma_<compressedproperty}. It follows that $|S(A)\setminus S(A')|\leq x_t-1\leq y_i-1$.
	
	Furthermore, since $i=s$,
	\begin{align*}S(A')\setminus S(A)=\{&(x_1,\dots,x_{i-1},a,x_{i+1},\dots,x_n):\\ &a\in\mathbb{Z}_{>0}\textnormal{ and replacing any coordinate by $0$ we get an element of $A'$}\}.
	\end{align*}
	
	Also, by our choice of $x$, any $z\not \in A$ with $z<y$ has $z_i=0$ and $z_l\not =0$ for $l\not =i$. But this easily shows that for $1\leq a\leq y_i-1$, the corresponding vector lies in $S(A')\setminus S(A)$.  So $|S(A')\setminus S(A)|\geq y_i-1\geq|S(A)\setminus S(A')|$.
	
	So we get a contradiction, finishing the proof.
\end{proof}

\begin{lemma}\label{lemma_gapminimiser}
	For every $n\geq 2$, initial segments $I$ of the balanced order minimise $\gap(I)$ among down-sets in $X_n$ of given size.
\end{lemma}

\begin{proof}
	Again we prove this by induction on $n$. The case $n=2$ is trivial, now assume $n\geq 3$ and the result holds for smaller values of $n$.
	Let $A$ be any subset of $X_n$, we show the initial segment of same size has a gap which is not bigger. Taking a down-set $A'$ in $X$ minimising $\sum_{x\in A'}\textnormal{(position of $x$ in the balanced order)}$ among sets with $|A'|=|A|$ and $\gap(A')\leq \gap(A)$, we may assume that $\CCC_i(A)=A$ for each $i$ (by Lemma \ref{lemma_<compression}). Suppose that there are $x,y\in X_n$ with $x<y$, $y\in A$ and $x\not \in A$. Take $y$ to be maximal and $x$ to be minimal (in the balanced order). Let $A'=A\setminus\{y\}\cup\{x\}$. Note that $A'$ is a down-set.
	
	By Lemma \ref{lemma_<compressedproperty}, there is a unique $s$ such that $x_s=0$. Then $\pi_j(A')\setminus \pi_j(A)=\emptyset$ if $j\not =s$ and $|\pi_s(A')\setminus \pi_s(A)|=1$. On the other hand, if $t$ is such that $y_t=0$ then $|\pi_t(A)\setminus \pi_t(A')|=1$. It follows that $\gap(A')\leq \gap(A)$, giving a contradiction.
\end{proof}

\begin{proof}[Proof of Theorem \ref{theorem_initialsegments}]
	Immediate from Lemma \ref{lemma_gapminimiser} and Lemma \ref{lemma_downset}.
\end{proof}

\bibliography{Bibliography}
\bibliographystyle{abbrv}	
	
\end{document}